\documentclass[letterpaper, 11pt,  reqno]{amsart}

\usepackage[margin=1.2in,marginparwidth=1.5cm, marginparsep=0.5cm]{geometry}

\usepackage{amsmath, mathrsfs, amssymb,amscd,amsthm,amsxtra, esint}
\usepackage{color}

\usepackage[implicit=true]{hyperref}

\usepackage{cases}
\allowdisplaybreaks[2]

\definecolor{gr}{rgb}   {0.,   0.69,   0.23 }
\definecolor{bl}{rgb}   {0.,   0.5,   1. }
\definecolor{mg}{rgb}   {0.85,  0.,    0.85}
\definecolor{yl}{rgb}   {0.8,  0.7,   0.}
\definecolor{or}{rgb}  {0.7,0.2,0.2}

\definecolor{olive}{rgb}{0.3, 0.4, .1}
\definecolor{fore}{RGB}{249,242,215}
\definecolor{back}{RGB}{51,51,51}
\definecolor{title}{RGB}{255,0,90}
\definecolor{dgreen}{rgb}{0.,0.6,0.}
\definecolor{gold}{rgb}{1.,0.84,0.}
\definecolor{JungleGreen}{cmyk}{0.99,0,0.52,0}
\definecolor{BlueGreen}{cmyk}{0.85,0,0.33,0}
\definecolor{RawSienna}{cmyk}{0,0.72,1,0.45}
\definecolor{Magenta}{cmyk}{0,1,0,0}
\definecolor{dorange}{RGB}{154,118,0}
\definecolor{corange}{RGB}{230, 176, 0}

\newtheorem{theorem}{Theorem} [section]

\newtheorem{lemma}[theorem]{Lemma}
\newtheorem{proposition}[theorem]{Proposition}
\newtheorem{remark}[theorem]{Remark}

\newtheorem{corollary}[theorem]{Corollary}

\DeclareMathOperator*{\intt}{\int}

\DeclareMathOperator*{\supp}{supp}

\DeclareMathOperator{\HS}{HS}

\DeclareMathOperator{\sgn}{sgn}

\newcommand{\noi}{\noindent}

\newcommand{\R}{\mathbb{R}}

\newcommand{\T}{\mathbb{T}}
\newcommand{\bul}{\bullet}

\newcommand{\PP}{\mathbb{P}}

\newcommand{\E}{\mathbb{E}}

\newcommand{\F}{\mathcal{F}}

\newcommand{\NB}{\mathbb{N}}

\newcommand{\dl}{\delta}

\newcommand{\nb}{\nabla}

\newcommand{\eps}{\varepsilon}

\newcommand{\s}{\sigma}
\newcommand{\Si}{\Sigma}
\newcommand{\ft}{\widehat}

\newcommand{\wt}{\widetilde}

\newcommand{\dx}{\partial_x}
\newcommand{\dt}{\partial_t}
\newcommand{\dd}{\partial}

\newcommand{\ta}{\theta}

\renewcommand{\o}{\omega}

\newcommand{\FL}{\mathcal{F}L} 

\newcommand{\les}{\lesssim}
\newcommand{\ges}{\gtrsim}

\newcommand{\jb}[1]
{\langle #1 \rangle}

\newcommand{\ind}{\mathbf 1}

\newcommand{\N}{\mathbb{N}}

\newcommand{\cX}{\mathcal{X}}
\newcommand{\cY}{\mathcal{Y}}

\newcommand{\Om}{\Omega}

\newtheorem*{ackno}{Acknowledgements}

\numberwithin{equation}{section}
\numberwithin{theorem}{section}

\makeatletter
\@namedef{subjclassname@2020}{\textup{2020} Mathematics Subject Classification}
\makeatother

\begin{document}
\baselineskip = 14pt

\title[UU for SKdV on $\R$]
{Unconditional well-posedness of the stochastic Korteweg-de Vries equation on the real line}

\author[D.~Greco, 
T.~Oh, and
K.~Tsugawa]
{Damiano Greco, 
Tadahiro Oh, 
and Kotaro Tsugawa}

\address{
Damiano Greco\\School of Mathematics\\
The University of Edinburgh\\
and The Maxwell Institute for the Mathematical Sciences\\
James Clerk Maxwell Building\\
The King's Buildings\\
Peter Guthrie Tait Road\\
Edinburgh\\ 
EH9 3FD\\
 United Kingdom}

\email{dgreco@ed.ac.uk}

%
%
%

%
\address{
Tadahiro Oh, 
School of Mathematics\\
The University of Edinburgh\\
and The Maxwell Institute for the Mathematical Sciences\\
James Clerk Maxwell Building\\
The King's Buildings\\
Peter Guthrie Tait Road\\
Edinburgh\\
EH9 3FD\\
 United Kingdom,
and  School of Mathematics and Statistics, Beijing Institute of Technology, Beijing 100081, China}

\email{hiro.oh@ed.ac.uk}

\address{Kotaro Tsugawa, 
Department of Mathematics, Faculty of Science and Engineering,
Chuo University, Bunkyo-ku, Tokyo, 112-8551, Japan}
 
 \email{tsugawa@math.chuo-u.ac.jp}

\subjclass[2020]{35Q53, 35R60, 60H15}

\keywords{stochastic Korteweg-de Vries equation;  Korteweg-de Vries equation;
uniqueness; Fourier-Lebesgue space}

\begin{abstract}

We study well-posedness issues 
of the stochastic Korteweg-de Vries equation (SKdV)
with an additive noise, posed 
on the real line.
By using the Fourier restriction norm method
adapted to the Fourier-Lebesgue space in time, 
we first prove global well-posedness
of SKdV
in $L^2(\R)$ 
without assuming the homogenous Sobolev
regularity, 
which was imposed in a work 
by de Bouard, Debussche, 
and Tsutsumi (1999).
Then, by adapting the argument by Zhou (1997)
to the stochastic setting, 
we
prove optimal pathwise unconditional uniqueness 
for SKdV 
in $L^2(\R)$.
In the appendix, we present a short argument
for proving  boundedness of the multiplication by a sharp cutoff function
in the  Fourier-Lebesgue and Sobolev spaces, 
which is of interest in its own right.

\end{abstract}

%
\maketitle
%


\tableofcontents



\newpage

\section{Introduction}
\label{SEC:1}

We consider 
the following stochastic Korteweg-de Vries equation (SKdV) with an additive noise, 
posed on the real line:
\begin{align}
\begin{cases}
\dt u + \dx^3 u + \dx (u^2) = \phi \xi\\
u|_{t = 0} = u_0, 
\end{cases}
\label{KDV1}
\end{align}

\noi
where 
\begin{align}
\phi \in \HS(L^2(\R); L^2(\R))
\label{HS1}
\end{align}

\noi
and $\xi$ is a (Gaussian) space-time white noise on $\R\times \R$
on a probability space $(\Om, \F, \PP)$
whose space-time covariance is (formally) given by 
\begin{align*}
 \E[ \xi(t_1, x_1)\xi(t_2, x_2) ] = \dl(t_1 - t_2) \dl (x_1 - x_2)
\end{align*} 

\noi
for $t_1, t_2 , x_1, x_2 \in \R$
with $\dl$ denoting the Dirac delta function.

Let us first go over the known well-posedness
results for 
the following (deterministic)
Korteweg-de Vries equation (KdV) on the real line:
\begin{align}
\dt u + \dx^3 u + \dx (u^2) =0, 
\label{KDV2}
\end{align}

\noi
corresponding to \eqref{KDV1} with $\phi = 0$.
In 
\cite{BO93}, Bourgain introduced 
the Fourier restriction norm method, involving the so-called $X^{s, b}$-spaces
defined via the norm: 
\begin{align*}
\| u\|_{X^{s,b}}=\| \jb{\xi}^s\jb{\tau -\xi^3 }^b\ft u (\tau, \xi)\|_{L^2_{ \tau, \xi} (\R\times \R)}, 
\end{align*}

\noi
and proved local well-posedness of \eqref{KDV2} in $L^2(\R)$, 
which immediately extended to global well-posedness
in $L^2(\R)$ thanks to the $L^2$-conservation for \eqref{KDV2}.
This result 
was then extended to $H^s(\R)$, 
$s > -\frac 34$, 
locally in time by 
Kenig, Ponce, and Vega~\cite{KPV96}
and globally in time 
by Colliander, Keel, Staffilani, Takaoka, and Tao~\cite{CKSTT03}.
In particular, in \cite{KPV96}, 
the authors proved the following bilinear estimate:
\begin{align}
\| \dx (u_1u_2)\|_{X^{s, b-1}}\les 
\prod_{j = 1}^2 \| u_j\|_{X^{s, b}}
\label{bilin0x}
\end{align}

\noi
for $s > - \frac 34$ with some $b > \frac 12$.
See \cite{NTT, Kishi, Guo}
for results on the endpoint case $s = -\frac 34$.
In a recent work 
 \cite{KV}, 
 Killip and Vi\c{s}an introduced the method of commuting flows, 
 exploiting the complete integrability of the equation, 
 and proved global well-posedness of \eqref{KDV2}
 in $H^{-1}(\R)$, which is sharp in view of  ill-posedness
 of \eqref{KDV2} in $H^s(\R)$, $s < -1$; see  \cite{Moli}.
We also mention~\cite{Tsu1} 
by the third author on well-posedness 
of KdV \eqref{KDV2}
with a (deterministic, time-independent) forcing term.

Let us now turn our attention to SKdV \eqref{KDV1}.
We say that $u$ is a solution to \eqref{KDV1}
if $u$ satisfies the following Duhamel formulation (= mild formulation):
\begin{align}
u(t) = S(t) u_0 - \int_0^t S(t-t') \dx (u^2)(t') dt' + \Psi, 
\label{mild1}
\end{align}

\noi
where $S(t) = e^{-t\dx^3}$ denotes
the linear KdV propagator
and $\Psi$ denotes the stochastic convolution: 
\begin{align}
\Psi(t)  = \int_0^t S(t-t') \phi \xi(dt'), 
\label{sto1}
\end{align}

\noi
representing the effect of the stochastic forcing in \eqref{KDV1}.
See Subsection \ref{SUBSEC:sto} for a precise meaning of
the stochastic integral on the right-hand side of \eqref{sto1}.
In \cite{DDT1}, 
de Bouard, Debussche, and Tsutsumi studied 
well-posedness issues of 
SKdV \eqref{KDV1}
by adapting the Fourier restriction norm method
(in particular, \cite{KPV93})
to the stochastic setting.
In particular, 
they 
 proved local well-posedness of~\eqref{KDV1}
in $H^s(\R)$
for $- \frac 58 < s \le 0$, 
 by assuming 
\begin{align}
\phi \in \HS(L^2(\R); H^s(\R))\cap
\HS(L^2(\R); \dot H^{ - \frac 38}(\R)).
\label{HS2}
\end{align}

\noi
They also proved global well-posedness
of \eqref{KDV1} in $L^2(\R)$
under the assumption \eqref{HS2} (with $s = 0$).
See also \cite{DDT2, Oh4, OQS}
for the known well-posedness results of SKdV in the periodic setting.

Given initial data $u_0 \in H^s(\T)$, 
local well-posedness of \eqref{KDV1}
means that there exists a unique solution 
$u$ to \eqref{KDV1} (or rather to \eqref{KDV2})
in the class
$C([0, T]; H^s(\T))$
for some 
almost surely positive time  $T = T_\o$, 
where uniqueness may hold in a smaller space
and the solution map is continuous in a suitable sense.
As such, the minimal and natural regularity 
assumption on the operator $\phi$ is 
\begin{align}
\phi \in \HS(L^2(\R); H^s(\R)).
\label{HS4}
\end{align}

\noi
See Lemma \ref{LEM:sto1}\,(ii).
In particular, the condition
\begin{align}
\phi \in 
\HS(L^2(\R); \dot H^{ - \frac 38}(\R))
\label{HS3}
\end{align}

\noi
appearing in \eqref{HS2} is somewhat artificial.
Since the work \cite{DDT1}, 
the question on the necessity of  the condition \eqref{HS3}
remained open over two decades.
In our first main result, 
we 
 establish global well-posedness
of \eqref{KDV1} in $L^2(\R)$
under \eqref{HS1}, 
thus showing that 
 the condition 
 \eqref{HS3} 
  imposed in \cite{DDT1}
  is in fact not necessary, at least for $s\ge 0$; see Remark \ref{REM:neg}.

\begin{theorem}\label{THM:1}
Let $s\ge 0$.
Suppose that \eqref{HS4} holds.
Then, 
the stochastic KdV equation
\eqref{KDV1} is globally well-posed in $H^s(\R)$.

\end{theorem}

In adapting the Fourier restriction norm method
to the stochastic setting, using the $X^{s, b}$-space, 
we need to take $b < \frac 12$ to capture the 
temporal regularity 
of the stochastic convolution~$\Psi$ in~\eqref{sto1}
(= regularity of a Brownian motion).
By adapting the bilinear estimate \eqref{bilin0x}
to the case $b < \frac 12$, there is an issue
in treating the 
``high $\times$ low $\to$ high'' interaction.
This was the reason that the assumption \eqref{HS2} 
(in particular, \eqref{HS3})
was introduced in~\cite{DDT1};
see \cite[Remark~2.1]{DDT1}.

In proving Theorem \ref{THM:1}, 
the main task is 
to prove local well-posedness
of \eqref{KDV1} under~\eqref{HS4}
(namely, 
{\it without} the assumption \eqref{HS3}).
Once this is achieved, global well-posedness
follows from the standard a priori bound (Lemma \ref{LEM:sto2})
as in \cite{DDT1}.
Our approach for proving local well-posedness
 is based on the Fourier restriction norm method
adapted to the Fourier-Lebesgue spaces in time;
given $1 \le q \le \infty$ and $b \in \R$, we define the Fourier-Lebesgue space
$\F L^{b, q}(\R)$ by the norm:
\begin{align*}
\|f\|_{\F L^{b, q}} = \| \jb{\tau}^b \ft f(\tau)\|_{L^q}.
\end{align*}

\noi
See \eqref{Xsb1} for the definition of the $X^{s, b}$-space
adapted to the Fourier-Lebesgue space in time.
The main observation is that the Fourier-Lebesgue
space $\F L^{b, q}(\R)$ captures the (local-in-time) regularity of
a Brownian motion, provided that 
\begin{align}
(b-1) q< -1.
\label{bilin0y}
\end{align} 

\noi
See \cite{BO}.
Note that by taking $q > 2$, we can choose $b > \frac 12$.
This observation allows us to capture the regularity of the stochastic convolution
$\Psi$ in \eqref{sto1} with $b > \frac 12$, 
thus enabling us to establish a key bilinear estimate;
 see Lemma \ref{LEM:sto1}\,(i)
and Proposition \ref{PROP:bilin1}.
See \cite{Oh4, FOW, OQS}, 
where a similar idea was used. 
We present a proof of Theorem \ref{THM:1} in Section \ref{SEC:GWP}.

\begin{remark}\label{REM:neg}\rm
(i) While 
local well-posedness of \eqref{KDV1} under \eqref{HS4} 
(without assuming~\eqref{HS3})
may be extended to negative Sobolev spaces, 
 we do not pursue this issue in this paper, 
since our main goal in this paper is to prove unconditional uniqueness
of \eqref{KDV1} in $L^2(\R)$ (Theorem~\ref{THM:2}).
We will address this issue (together with global well-posedness
of \eqref{KDV1} in negative Sobolev spaces)
in a forthcoming work.

\smallskip

\noi
(ii) 
In \cite{CLO}, Cheung, Li, and the second author
introduced a new approach for proving global well-posedness of 
stochastic dispersive PDEs
by 
adapting the $I$-method (= method of almost conservation laws)
to the stochastic setting.
It would be of interest to apply the approach in \cite{CLO}
to prove global well-posedness of SKdV \eqref{KDV1}
in negative Sobolev spaces.

\end{remark}


Next, we turn our attention to the uniqueness issue for \eqref{KDV1}.
We first recall 
from~\cite{KATO} 
that  a Cauchy problem is  said to be {\it unconditionally (locally) well-posed}
in $H^s(\R)$
if for every initial condition $u_0 \in H^s(\R)$,
there exist $T >0$ and a unique solution $u \in C([0, T ];H^s(\R))$
with  $u|_{t = 0} = u_0$.
We refer to such uniqueness in the entire class $C([0, T ];H^s(\R))$, 
without intersecting with any auxiliary function space, as {\it unconditional uniqueness}.
Unconditional uniqueness is a concept of uniqueness which does not depend
on how solutions are constructed
and thus is of fundamental importance.
As such, 
unconditional uniqueness for (deterministic) dispersive PDEs 
 has been studied extensively
by various methods;
see, for example, 
 \cite{Furioli, KO, GKO, 
 KOPV, CGKO, Chen, HS, OW2, Kishi21, KT1, KT2, BOW}.
In the stochastic setting, however, 
there is no unconditional uniqueness result
(for low regularity initial data).

Due to the use of the Fourier restriction norm method, 
the uniqueness
of 
 a solution constructed
in Theorem~\ref{THM:1}
holds (locally in time) in 
\[X^{0, b}_q(T) \cap C([0, T]; L^2(\R))\]

\noi
for some $b > \frac 12$ and $q > 2$, satisfying \eqref{bilin0y}
(see \eqref{lwp0}), 
where $X^{s, b}_q(T)$ is as in \eqref{Xsb2}.
Namely, 
 the uniqueness of a solution constructed in Theorem \ref{THM:1} holds 
 {\it conditionally}.
In the next theorem, we establish 
a sharp unconditional uniqueness
result for \eqref{KDV1}
in $L^2(\R)$.

\begin{theorem}\label{THM:2}
Suppose that \eqref{HS1} holds.
Then, pathwise unconditional uniqueness
in $L^2(\R)$ 
holds for 
the stochastic KdV equation
\eqref{KDV1} on $\R$.
More precisely, 
given two global-in-time solutions $u_1, u_2 \in C(\R; L^2(\R))$
to \eqref{mild1}, 
there exists 
$\Si \subset \Om$ with $\PP (\Si) = 1$ such that 
for each $\o \in \Si$, 
we have $u_1(\o) \equiv u_2(\o)$, 
where $u_j(\o)$, $j = 1, 2$,  denotes 
the  solution to \eqref{mild1}
with a realization of the noise for $\o \in \Si$.

\end{theorem}

We note that a global-in-time solution $u \in C(\R; L^2(\R))$ to \eqref{mild1}
means that there exists a set $\Si_u \subset \Om$
with $\PP(\Si_u) = 1$
such that, for each $\o \in \Si_u$,  $u(\o)$
satisfies \eqref{mild1} 
in $C(\R; L^2(\R))$
with a realization of the noise (and hence 
of the stochastic convolution $\Psi = \Psi(\o)$).
Namely, the set $\Si_u$ of full probability depends on the solution $u$.
In discussing unconditional uniqueness, we do not know how a given solution $u$
has been constructed and thus there is no information 
about the associated set $\Si_u$ of full probability.
In particular, the intersection of 
$\Si_u$ over all possible solutions $u$ to \eqref{mild1} may not have full probability
(and may even have probability $0$).
This is the reason that, in Theorem \ref{THM:2}, 
we first start with two given solutions $u_j$, $j = 1, 2$, 
and then determine a set $\Si = \Si_{u_1} \cap \Si _{u_2}$ of 
full probability, where we compare the solutions $u_j$, $j = 1, 2$; see the discussion at the beginning of Subsection \ref{SUBSEC:4.3}.

Theorem \ref{THM:2}
establishes the first (pathwise) unconditional uniqueness
result for
stochastic dispersive PDEs in a low regularity setting.
We  point out that Theorem \ref{THM:2} is optimal
since the $L^2$-regularity is required to make
sense of the quadratic nonlinearity 
$\dx (u^2)(t)$
in \eqref{KDV1} (and~\eqref{mild1})
as a spatial distribution for given $t \ge 0$
(without intersecting with an auxiliary function space).

In \cite{Zhou}, Zhou
 proved unconditional uniqueness
 of the deterministic KdV \eqref{KDV2}
in $L^2(\R)$
by establishing a modified version of  the bilinear estimate \eqref{bilin0x}.
Our proof is based on adapting Zhou's argument
to the stochastic setting.
We present a proof of Theorem \ref{THM:2}
in Section~\ref{SEC:UU}.

\begin{remark}\rm
(i) By adapting the normal form approach in \cite{BIT}
to the stochastic setting, 
Sosoe, Wang, and the first two authors 
\cite{GOSW} recently established
unconditional uniqueness 
in $L^2(\T)$
of SKdV \eqref{KDV1} on the circle. 

\smallskip

\noi
(ii)
In a recent work
\cite{CLOZ}, 
Chapouto, Li, Zhao, and  the second author 
studied shallow-water convergence of the (scaled) intermediate
long wave equation to KdV.
As one of key ingredients in the proof, they
implemented an infinite iteration of normal form reductions
for 
the (deterministic) KdV \eqref{KDV2} on $\R$.
We expect that a combination of the argument 
in \cite{CLOZ, GOSW} would provide
an alternative proof of Theorem \ref{THM:2}.
We, however,   point out  that 
such a normal form approach would be much lengthier
and thus we 
do not pursue this issue here.

\end{remark}

\section{Notations,  function spaces, and tools from stochastic analysis}

\subsection{Basic notations}
Let $A\les B$ denote an estimate of the form $A\leq CB$ for some constant $C>0$. We write $A\sim B$ if $A\les B$ and $B\les A$, while $A\ll B$ denotes $A\leq c B$ for some small constant $c> 0$. 

Let $\eta \in C_c(\R; [0, 1])$ be 
a smooth bump function (in time)
with $\supp \eta \subset [-1, 2]$
such that $\eta \equiv 1$ on $[0, 1]$.
Given $T > 0$, let 
\begin{align}
\eta_{_T}(t) = \eta(T^{-1}t).
\label{eta1}
\end{align}

\noi
Then,  when $b \ge 0$, we have
\begin{align}
\| \eta_{_T}\|_{\F L^{b, q}}
\les T^{- b + \frac{1}{q'}} 
\|\eta \|_{\F L^{b, q}}
\label{FL1a}
\end{align}

\noi
for $0 \le T \le 1$, where $q'$ is the H\"older conjugate of $q$.

In dealing with function spaces
on space-time functions, 
we  use short-hand notations such as
$C_T H^s_x  = C\big([0, T]; H^s(\R))$, 
 when there is no ambiguity.

\subsection{Function spaces}\label{SUBSEC:2.2}

In order to prove well-posedness in $L^2(\R)$ of SKdV \eqref{KDV1}
under the assumption \eqref{HS1} (namely, 
without imposing a further regularity on $\phi$ in \eqref{KDV1}), 
we introduce
a version of 
the $X^{s, b}$-space
adapted to the Fourier-Lebesgue space $\FL^{b, p}(\R)$ in time;
 see \cite{GH, FOW}.
 See also
\cite{Grun1, Grun2}.
Given $s, b\in\R$ and $1 \le  q \le  \infty$,
define the $X^{s,b}_q$-space via the norm:
\begin{align}
\| u\|_{X^{s,b}_q}=\| \jb{\xi}^s\jb{\tau -\xi^3 }^b\ft u (\tau, \xi)\|_{L^2_\xi L^q_{\tau}(\R\times \R)}, 
\label{Xsb1}
\end{align}

\noi
where $\jb{\,\cdot\,} = (1 + |\cdot|^2)^\frac 12$.
Given $T > 0$, 
we define the local-in-time version of the $X^{s, b}_q$-space
by setting
 \begin{align}
\| u\|_{X^{s,b}_{q}(T) }=\inf\big\{\| v\|_{X^{s,b}_{q}}: v|_{[0,T]}=u\big\}, 
\label{Xsb2}
\end{align}
     
\noi
where the infimum is taken over all extensions $v$ of $u$
from $[0, T]$ to $\R$.

\smallskip
\begin{itemize}
\item[(i)]
\underline{$b q'> 1$:} 
In this case, we have 
\begin{align}
 X^{s, b}_{q}(T)\subset C([0, T]; H^s(\R)).
 \label{Xsb1a}
\end{align}

\noi
See \cite[(2.3)]{FOW}.

\smallskip
\item[(ii)]
\underline{$bq' < 1$:}
In this case, 
the embedding \eqref{Xsb1} does not hold.
For $b \ge 0$ and $1 < q < \infty$ with $bq' < 1$, 
it follows from Lemma \ref{LEM:X2} with \eqref{Xsb2} that 
\begin{align}
\| u\|_{X^{s,b}_{q}(T) } 
\sim \| \ind_{[0, T]}\cdot u\|_{X^{s,b}_{q}}.
\label{Xsb3}
\end{align}

\noi
Then, given $0 \le b < b_0$ 
and  $1 < q < \infty$
with $b_0 q' < 1$, 
a slight modification of the proof of \cite[Lemma~3.2]{CO}
with~\eqref{Xsb3}
yields
\begin{align}
 \| u\|_{X^{s, b}_{q}(T)} 
 \les  T^{\frac 1{q'} ( 1- \frac{b}b_0)}\| u\|_{X^{s, b_0}_{q}(T)} 
\label{Xsb4}
\end{align}

\noi
for $0 \le T\le 1$, which in particular implies
that, given $T_0 > 0$, we have 
\begin{align}
\lim_{T \to 0} \| u\|_{X^{s, b}_{q}(T)} 
= 0
\label{Xsb5}
\end{align}

\noi
for any function $u \in 
 X^{s, b_0}_{q}(T_0)$.
 Note that 
we do not need to assume 
 $u|_{t = 0} = 0$ in~\eqref{Xsb5}.
Compare this with the case $bq' > 1$, 
where \eqref{Xsb5} fails unless  $u|_{t = 0} = 0$.
 

\end{itemize}


\smallskip

We 
recall the standard linear estimates.
See 
\cite[Lemma 2.2]{FOW}
for the proof.
The bound~\eqref{lin1} follows from \eqref{FL1a}.

\begin{lemma}\label{LEM:lin1}
\textup{(i) (homogeneous linear estimate).}
Given $1\leq q\leq \infty$  and $s,b\in \R$, we have 
\begin{equation*}
\|S(t)f\|_{X^{s,b}_{q}(T)} \les \|f\|_{H^s}
\end{equation*}

\noi
for any $0 < T \leq 1$.
Moreover, for $0 \le b < \frac 1{q'}$, we have 
\begin{align}
\|S(t)f\|_{X^{s,b}_{q}(T)} \les T^{- b + \frac1 {q'}}\|f\|_{H^s}
\label{lin1}
\end{align}

\noi
for any $0 < T \leq 1$.

\medskip

\noi
\textup{(ii) (nonhomogeneous linear estimate).}
Let $s\in \R$,  $1<q<\infty$ and $-\frac{1}{q}<b'\leq 0\leq b \leq 1+b'.$ 
Then, we have 
\begin{equation*}
\bigg\| \int_0^tS(t-t')F(t')dt'\bigg\|_{X^{s,b}_{q }(T)}\les T^{1+b'-b}\|F\|_{X^{s,b'}_{q}(T)}
\end{equation*}

\noi
for any $0< T \leq 1$.

\end{lemma}


\subsection{Tools from stochastic analysis}
\label{SUBSEC:sto}



We first provide a precise meaning to the expression in \eqref{sto1}.
Fix a (real-valued) orthonormal basis $\{e_n \}_{n \in \N}$ of $L^2(\R)$.
Define a sequence 
$\{ B_n \}_{n \in \N}$  by 
$B_n(t) = \jb{\xi, \ind_{I_t} \cdot e_n}_{ t, x}$, 
where $I_t = [0, t]$ 
for $t \ge 0$ and $I_t = [t, 0]$ for $t < 0$.
Here, $\jb{\cdot, \cdot}_{ t, x}$ denotes 
the duality pairing on $\R\times \R$.
Note that   we have
 \[\E[ B^2_n(t)] = \E\Big[
 \jb{\xi, \ind_{I_t} \cdot e_n}_{t, x}\jb{\xi, \ind_{I_t} \cdot e_n}_{t, x}
 \Big] = \|\ind_{I_t} \cdot e_n\|_{L^2_{t, x}}^2 = |t|\]

\noi
for any 
$n \in \N$, and 
$\E[ B_m(t)B_n(t)] = 0$ 
if $m \ne n$.
As a result, 
we see that $\{ B_n \}_{n \in \N}$ is a family of mutually independent 
(two-sided, real-valued)
Brownian motions.
Then, by  defining
a cylindrical Wiener process $W$ on $L^2(\R)$ by 
\begin{align*}
 W(t)  =  \sum_{n\in \NB}
B_n (t) e_n, 
\end{align*}

\noi
we can express the  stochastic convolution $\Psi$ in~\eqref{sto1}
as
\begin{align*}
\Psi(t) =   \int_0^t S(t - t') \phi dW (t')
=   \sum_{n \in \NB} 
\int_0^t S(t - t') \phi(e_n)  dB_n (t'), 
\end{align*}

\noi
where each integral in the summation is interpreted
as a Wiener integral.

The following lemma summarizes  regularity properties
of the stochastic convolution $\Psi$.

\begin{lemma}\label{LEM:sto1}
Given $s \in \R$, 
let $\phi \in \HS(L^2(\R); H^s(\R))$
and  $\Psi$ be as in \eqref{sto1}.
Fix $T >0$.

\smallskip

\noi
{\rm (i)}
Let $b \in \R$ and $1 \le q < \infty$  such that 
\begin{align}
(b-1) q< -1.
\label{BM1}
\end{align}

\noi
Then, 
given  any finite $p  \ge 1$, there exists $C = C (T,p,  s, b, q) >0$ such that
\[
\E \bigg[  \| \Psi  \|_{X^{s, b}_q(T)}^p \bigg] 
\le C \| \phi \|_{\HS(L^2; H^s)}^p.
\]

\smallskip

\noi
{\rm (ii)}
We have $\Psi \in C ( [0,T];H^s (\R))$ almost surely.
Moreover, given any finite $p  \ge 1$, there exists $C = C (T, p, s) >0$ such that
\begin{align*}
\E \bigg[ \sup_{0 \le t \le T} \| \Psi (t) \|_{H^s}^p \bigg] 
\le C \| \phi \|_{\HS(L^2; H^s)}^p.
\end{align*}

\end{lemma}

We recall that the condition \eqref{BM1}
captures the (local-in-time) Fourier-Lebesgue regularity of a Brownian motion;
see
\cite{BO}.
We also note that \eqref{BM1} is equivalent to 
$b q' <  1$, and thus the embedding \eqref{Xsb1a} does not hold.

\begin{proof}
Part (i)  follows from a standard argument; see, for example, 
\cite[Proposition 2.1]{DDT2}
and 
\cite[Lemma 3.1]{FOW}.
See also 
\cite{DDT1, Oh4}.
Part  (ii) follows from a standard application of the factorization method
or a Kolmogorov continuity criterion argument; see, 
for example,  \cite[Chapter 5]{DZ}.
\end{proof}

Next, we 
recall the following a priori bound on 
the $L^2$-norm of a solution to \eqref{KDV1}.

\begin{lemma} \label{LEM:sto2}
Given $u_0 \in L^2(\R)$
and $\phi \in \HS(L^2(\R); L^2(\R))$, 
 let $u$ be a  solution to~\eqref{KDV1}
 with $u|_{t = 0} = u_0$.
 Then, given a target time $T_0>0$, 
there exists $C_1 = C_1 \big(\|u_0\|_{L^2}, T_0, \| \phi \|_{\HS (L^2; L^2)}\big)>0$ such that for any stopping time $T$ with $0<T< \min (T^{\ast}, T_0)$ almost surely, we have
\begin{align}
\E \bigg[ \sup_{0\le t \le T} \|u(t)\|_{L^2_x}^2 \bigg] \le C_1,
\label{app1}
\end{align}

\noi
where 
 $T^{\ast} = T^{\ast}_{\omega} (u_0, \phi)$ is the forward maximal time of existence.

\end{lemma}

The a priori bound \eqref{app1}  follows
from Ito's lemma (applied to $\|u(t)\|_{L^2_x}^2$) and the Burkholder-Davis-Gundy inequality.
In order to justify an application of Ito's lemma, 
one needs to go through a certain approximation argument.
See \cite[Subsection 3.2]{DDT1}.
See also   \cite[Proposition~3.2]{DD1}.
Since the required argument is standard, 
we omit details.

\section{Well-posedness of SKdV in $L^2(\R)$}
\label{SEC:GWP}

In this section, we present a proof of Theorem \ref{THM:1}.
Our main task is to prove local well-posedness
of \eqref{KDV1} {\it without} assuming \eqref{HS3}.
We achieve this goal by working with 
the $X^{s, b}_q$-space in \eqref{Xsb1}
with $b > \frac 12$ and $q>2$, satisfying 
\eqref{BM1}.

We state a key bilinear estimate in 
the $X^{s, b}_q$-spaces.

\begin{proposition}\label{PROP:bilin1}
Let $s \ge 0$, 
$b_0 \le - \frac 14$, 
and $b > \frac 12 $ and $2 <  q \le \infty$ satisfy
\begin{align}
\frac 14 + \frac1{2q'}
< b < \frac{1}{q'}.
\label{bilin0}
\end{align}

\noi
Then, we have 
\begin{align}
\| \dx (u_1u_2)\|_{X^{s, b_0}_q(T)}\les 
\prod_{j = 1}^2 \| u_j\|_{X^{s, b}_q(T)}, 
\label{bilin1}
\end{align}

\noi
uniformly in $T > 0$.
\end{proposition}

By assuming
Proposition \ref{PROP:bilin1}, 
we first present a proof of Theorem \ref{THM:1}
on global well-posedness of \eqref{KDV1} in $H^s(\R)$, $s \ge 0$
under the assumption \eqref{HS4}.

\begin{proof}[Proof of Theorem \ref{THM:1}]

We only consider the case $s = 0$.

We first discuss local well-posedness.
Given small $\eps > 0$, 
let 
\begin{align}
b = \frac 12 + \eps \qquad \text{and}\qquad  
q = \frac{2}{1-3\eps} > 2
\label{lwp0}
\end{align}

\noi
such that 
\eqref{bilin0}
(and thus 
\eqref{BM1}) holds.
With this choice of $q > 2$, choose $b_1 \in \R$ such that 
\begin{align}
\frac 1 {q'}< b_1 \le \frac 34-\eps,  
\label{lwp1}
\end{align}

\noi
and set $b_0 = b_1 - 1 + \eps$.
By applying a contraction argument with Lemma \ref{LEM:lin1},
Lemma \ref{LEM:sto1}\,(i), and 
Proposition~\ref{PROP:bilin1} (see also \eqref{lwp2}), 
we see that there exists almost surely a unique solution 
$u \in X^{0, b}_q(T)$
to~\eqref{KDV1}
on a time interval $[0, T]$
for some 
almost surely positive time  $T = T_\o$.

From \eqref{bilin0}, we haver  $b q' < 1$ and thus the embedding \eqref{Xsb1a}
does not hold.
Hence, we need to show that each term on the right-hand side of \eqref{mild1}
belongs to 
$C([0, T]; H^s(\R))$
in order to conclude that the solution 
$u \in X^{0, b}_q(T)$ constructed above
indeed belongs to 
$C([0, T]; H^s(\R))$.
By standard analysis  and Lemma \ref{LEM:sto1}\,(ii), 
we see that the first and third terms 
on the right-hand side of \eqref{mild1}
belong to 
$C([0, T]; H^s(\R))$, almost surely.

Let $b_1$ be as in \eqref{lwp1}.
Then,  from Lemma \ref{LEM:lin1} and Proposition \ref{PROP:bilin1}, 
we have
\begin{align}
\bigg\| \int_0^t S(t-t') \dx (u^2)(t') dt'\bigg\|_{X^{0, b_1}_q(T)}
\les 
\| u\|_{X^{s, b}_q(T)}^2, 
\label{lwp2}
\end{align}

\noi
which is almost surely finite.
Then, it follows from the embedding \eqref{Xsb1a} with \eqref{lwp1}
that 
 the second  term
on the right-hand side of \eqref{mild1}
belongs to 
$C([0, T]; H^s(\R))$, almost surely.
Therefore, we conclude that the solution 
$u\in X^{0, b}_q(T)$
to~\eqref{KDV1}
belongs to 
$C([0, T]; H^s(\R))$, almost surely.
This proves local well-posedness of \eqref{KDV1}.

Next, we briefly discuss how to obtain global well-posedness.
From the local well-posedness argument
and Lemma \ref{LEM:sto1}\,(ii), 
we have  the following blowup alternative;
let $T^{\ast} = T^{\ast}_{\omega} (u_0, \phi)$ be the forward maximal time of existence.
Then, either
\begin{align}
T^* = \infty
\qquad \text{or} \qquad
\lim_{t \nearrow T^*} \| u(t) \|_{L^2_x} = \infty.
\label{lwp3}
\end{align}

\noi
By applying
Lemma \ref{LEM:sto2}, we see that the second scenario in \eqref{lwp3}
does not occur, almost surely.
Namely, $T^* = \infty$ almost surely, 
thus yielding global well-posedness of \eqref{KDV1}.
\end{proof}

We now present a proof of 
Proposition \ref{PROP:bilin1}.

\begin{proof}[Proof of Proposition \ref{PROP:bilin1}]
It follows from \eqref{Xsb2} that 
 \eqref{bilin1} follows once we prove the corresponding estimate without a time restriction:
\begin{align}
\| \dx (u_1u_2)\|_{X^{s, b_0}_q}\les 
\prod_{j = 1}^2 \| u_j\|_{X^{s, b}_q}.
\label{bilin2}
\end{align}

We prove \eqref{bilin2} by following the proof of 
\cite[Theorem 1.1]{KPV96} for $s = 0$, 
in particular, the proof of 
\cite[Lemma 2.4]{KPV96}.
In view of the triangle inequality: $\jb{\xi}^s \les \jb{\xi_1}^s + \jb{\xi_2}^s$
for $\xi = \xi_1 + \xi_2$ when $s \ge 0$, 
it suffices to prove \eqref{bilin2} for $s = 0$.
Moreover, 
\eqref{bilin2} for $s = 0$
follows once we prove the following estimate:
\begin{align}
\bigg\|
\intt_{\substack{\tau = \tau_1 + \tau_2\\\xi = \xi_1 + \xi_2}}
\frac{|\xi|} {\s^{-b_0}}
\frac{ f_1(\tau_1, \xi_1)f_2(\tau_2, \xi_2) d\tau_1 d\xi_1}
{\s_1^{b}
\s_2^{b}}
\bigg\|_{L^2_\xi L^q_\tau}
\les \prod_{j = 1}^2 \|f_j\|_{L^2_\xi L^q_\tau}, 
\label{bilin3}
\end{align}

\noi
where $\s$, $\s_1$, and $\s_2$ are given by 
\begin{align}
\s = \jb{\tau -\xi^3}
\qquad \text{and}\qquad 
\s_j = \jb{\tau_j - \xi_j^3}, \quad j = 1, 2.
\label{bilin4}
\end{align}

\noi
For simplicity of notation, 
we  drop
the conditions 
$\tau = \tau_1 + \tau_2$
and $\xi = \xi_1 + \xi_2$
in the following 
but it is understood that these conditions hold.

By
H\"older's inequality followed by 
 Minkowski's integral inequality, 
 we have 
\begin{align}
\begin{split}
\text{LHS of }\eqref{bilin3}
& \le \sup_{\tau, \xi \in \R}  M_{\tau, \xi} \, 
\prod_{j = 1}^2 \|f_j\|_{L^2_\xi L^q_\tau},
\end{split}
\label{bilin4a}
\end{align}

\noi
where $M_{\tau, \xi}$ is given by 
\begin{align*}
 M_{\tau, \xi} = \bigg\|
\frac{ |\xi|} {\s^{-b_0}\s_1^{b}
\s_2^{b}}
\bigg\|_{L^2_{\xi_1}L^{q'}_{\tau_1}}.
\end{align*}

From 
\cite[Lemma 4.2]{GTV}
with \eqref{bilin0}, 
we have 
\begin{align}
M_{\tau, \xi}^2 
\les 
\frac{ |\xi|^2} {\s^{-2b_0}}
\int
\frac{d\xi_1}{\jb{\tau - \xi^3 + 3 \xi \xi_1 \xi_2}^{4b - \frac 2{q'}}}.
\label{bilin6}
\end{align}

\noi
We now proceed with a change of variables as in 
 the proof of 
\cite[Lemma 2.4]{KPV96}.
Namely, by letting
\begin{align}
\mu =  \mu_{\tau, \xi}(\xi_1) = \tau - \xi^3 + 3 \xi \xi_1 (\xi - \xi_1), 
\label{cv0}
\end{align}

\noi
we have 
\begin{align}
 \dd_{\xi_1} \mu = 3\xi(\xi - 2\xi_1).
\label{cv1}
\end{align}

\noi
On the other hand, by expressing $\xi_1$ in terms of $\mu$, we have\footnote{Note that \eqref{cv0}
guarantees that what's inside the square root in \eqref{cv2} 
is non-negative.}
\begin{align}
\xi_1 = \frac {\xi}{2} \pm \frac 12 \sqrt{\frac{4\tau - \xi^3- 4\mu }{3\xi}}.
\label{cv2}
\end{align}

\noi
Thus, from \eqref{cv1} and \eqref{cv2}, we have
\begin{align}
|\dd_{\xi_1}\mu | = \sqrt{ (3\xi) (4\tau - \xi^3- 4\mu)}
\label{cv3}
\end{align}

\noi
Hence, from \eqref{bilin6}, \eqref{cv0}
and \eqref{cv3}
with \cite[(2.9) in Lemma 2.3]{KPV96}, we have 
\begin{align}
\begin{split}
M_{\tau, \xi}^2 
& \les 
\frac{ |\xi|^\frac 32} {\s^{-2b_0}}
\int
\frac{d\mu}{\jb{\mu}^{4b - \frac 2{q'}}
 \sqrt{ |4\tau - \xi^3- 4\mu|}}\\
&  \les 
\frac{ |\xi|^\frac 32} {\jb{\tau - \xi^3}^{-2b_0}
\jb{4\tau - \xi^3}^\frac{1}{2}}
\les 1, 
\end{split}
\label{cv4}
\end{align}

\noi
uniformly in $\tau, \xi \in \R$, 
provided that 
\begin{align*}
4b - \frac 2{q'} > 1\qquad \text{and}\qquad 
b_0 \le - \frac 14.
\end{align*}

\noi
Hence, 
the desired bound \eqref{bilin3}
follows from \eqref{bilin4a} and \eqref{cv4}.
\end{proof}


\vspace{1cm}

\section{Unconditional uniqueness in $L^2(\R)$}
\label{SEC:UU}

In this section, we prove unconditional uniqueness of SKdV \eqref{KDV1}
in $L^2(\R)$ (Theorem~\ref{THM:2})
by adapting Zhou's argument \cite{Zhou} to the current stochastic setting.

\subsection{Preliminary analysis}

Let $u \in C([0, 1]; L^2(\R))$ 
be a solution to KdV \eqref{KDV2}
with $\|u\|_{C_1 L^2_x} \le \dl$
for some small $\dl > 0$.
Then, by applying scaling and interpolating a trivial bound on the $X^{0, 0}([0, 1])$-norm
and a bound on the $X^{- \frac 32-\eps, 1}([0, 1])$-norm
(which was obtained by 
using the equation; see \eqref{ZH4} below), 
Zhou \cite{Zhou} showed
that 
the following bound holds:
\begin{align*}
\|u\|_{X^{- \frac 32\ta-\eps, \ta}([0, 1])} \les \dl
\end{align*}

\noi
for any $0 \le \ta \le 1$ and $\eps > 0$;
see \cite[(1.16)]{Zhou}.

In the current stochastic setting, 
we can not directly use the equation 
to estimate 
the $X^{- \frac 32-\eps, 1}(T)$-norm
of a solution  $u$ to SKdV \eqref{KDV1},
since the right-hand side of \eqref{KDV1}
does not make sense pointwise in time.
In order to overcome this issue, we proceed with the following first order 
expansion  \cite{McK, BO96, DPD}:
\begin{align}
u = \Phi + v.
\label{ZH2}
\end{align}

\noi
Here, $\Phi = \Phi(u_0)$ is defined by
\begin{align}
\Phi(t) = S(t) u_0 + \Psi(t), 
\label{ZH3}
\end{align}

\noi
where $\Psi$ denotes the stochastic convolution in \eqref{sto1}.
Then, $v = u - \Phi$ satisfies
\begin{align}
\begin{cases}
\dt v + \dx^3 v + \dx \big((v+\Phi) ^2\big) = 0\\
v|_{t = 0} = 0.
\end{cases}
\label{KDV3}
\end{align}

\begin{remark}\rm 
We included the linear solution $S(t) u_0$ 
in the definition \eqref{ZH3} of $\Phi$
such that $v|_{t = 0} = 0$,  
which allows us to avoid
a scaling argument; see \eqref{ZH9} below.
We note that applying the KdV scaling to the unknown $u$
and the white noise scaling to the space-time white noise $\xi$ in \eqref{KDV1} 
preserves SKdV \eqref{KDV1} only in law, 
which is not necessarily suitable for a pathwise argument.
We refer interested readers to 
\cite[Subsection 3.3]{CLO}
for a scaling in the stochastic setting.
\end{remark}

Let $v\in C(\R; L^2(\R)) $ be a solution to \eqref{KDV3}.
Namely, there exists a set $\Si \subset \Om$ with $\PP(\Si) = 1$
such that, 
for  $\o \in \Si$, 
there exists a global solution $v = v(\o)
\in C(\R; L^2(\R))$ to \eqref{KDV3}
with a realization of the noise (and hence of $\Phi = \Phi(\o)$ in \eqref{KDV3})
for $\o \in \Si$.
In the following, we fix $\o \in \Si$
and suppress the dependence on $\o \in \Si$.

Given $\eps > 0$, it follows from 
 Plancherel's identity,~\eqref{KDV3}, 
and 
Cauchy-Schwarz' inequality (on the Fourier side and
then on the physical side) that 
\begin{align}
\begin{split}
\| v \|_{X^{-\frac 32 - \eps, 1}}
& \sim 
 \| v\|_{L^2_t H^{-\frac 32 - \eps}_x}
+   \| (\dt - \dx^3) v\|_{L^2_t H^{-\frac 32 - \eps}_x}\\
& \les
 \| v\|_{L^2_t L^2_x}
+ \|(v+\Phi)^2\|_{L^2_t H^{-\frac 12 - \eps}_x}\\
& \les
 \| v\|_{L^2_t L^2_x}
+ \|v\|_{L^2_t L^2_x}^2
+ \|\Phi\|_{L^2_t L^2_x}^2.
\end{split}
\label{ZH4}
\end{align}

\noi
See \cite[(1.10)-(1.14)]{Zhou}.
If we were to follow Zhou's argument, 
we would first apply 
scaling such that, 
when restricted to the time interval $[0, 1]$, 
the contribution from 
the right-hand side of \eqref{ZH4}
is small.
In the following, however, 
we proceed slightly differently from \cite{Zhou}
so that we do not rely on  a scaling.
We point out that by working on a short time interval $[0, T]$
for some small $T > 0$, 
we would get a negative power of $T$ 
in \eqref{ZH4},
and thus some care is needed.
See \eqref{ZH7} and \eqref{ZH9} below.

Fix small $\dl > 0$ and let $0 < T \le \frac 12 $ (to be chosen later).
Let $\eta_{_T}$ be as in \eqref{eta1}.
Note that  $\supp (\eta_{_T}) \subset [-1, 1]$.
From 
\eqref{Xsb2}, 
\eqref{ZH2}, \eqref{ZH3}, and 
Lemma \ref{LEM:sto1}\,(i), and 
Lemma \ref{LEM:sto2} (with $T = 1$), we have 
\begin{align}
\begin{split}
\| v \|_{X^{0, 0}(T)} 
& \le 
\| \eta_{_T} v \|_{X^{0, 0}(T)} = \|\eta_{_T} v\|_{L^2_{T, x}}\\
& \le
T^\frac{1}{2} \|v\|_{C([-1, 1]; L^2_{x})}\\
& \le T^\frac 12 C\big(\|u_0\|_{L^2}, \|\phi\|_{\HS(L^2; L^2)}\big)\\
& \le \dl, 
\end{split}
\label{ZH5}
\end{align}

\noi
provided that 
$T = T\big(\o, \dl, \|u_0\|_{L^2}, \|\phi\|_{\HS(L^2; L^2)}\big) > 0$ is sufficient small.
Proceeding as in \eqref{ZH4}, 
we have
\begin{align}
\begin{split}
\| v \|_{X^{-\frac 32 - \eps, 1}(T)}
& \les 
 \|\eta_{_T} v\|_{L^2_t H^{-\frac 32 - \eps}_x}
+   \| (\dt - \dx^3) (\eta_{_T}v)\|_{L^2_t H^{-\frac 32 - \eps}_x}\\
& \les
 \| \eta_{_T} v\|_{L^2_t L^2_x}
+ \|\eta_{_T}(v+\Phi)^2\|_{L^2_t H^{-\frac 12 - \eps}_x}
+ \|\dt \eta_{_T}\cdot v\|_{L^2_t H^{- \frac{3}{2}-\eps}_x}
\\
& \les
 T^\frac 12 \| v\|_{C([-1, 1]; L^2_x)}
+ T^\frac 12\Big( \|v\|_{C([-1, 1]; L^2_x)}^2
+  \|\Phi\|_{C([-1, 1]; L^2_x)}^2\Big)\\
& \quad + T^\frac 12\|\dt \eta_{_T}\cdot v\|_{C([-T, 2T]; H^{-\frac{3}{2}-\eps}_x)}.
\end{split}
\label{ZH7}
\end{align}

\noi
By taking 
$T = T\big(\o, \dl, \|u_0\|_{L^2}, \|\phi\|_{\HS(L^2; L^2)}\big) > 0$ 
sufficiently small 
 with Lemma \ref{LEM:sto1}\,(ii)
and Lemma \ref{LEM:sto2}, 
we have 
\begin{align}
 T^\frac 12 \| v\|_{C([-1, 1]; L^2_x)}
+ T^\frac 12  \Big(\|v\|_{C([-1, 1]; L^2_x)}^2
+  \|\Phi\|_{C([-1, 1]; L^2_x)}^2\Big)
\les \dl.
\label{ZH8}
\end{align}

\noi
From \eqref{eta1} and the Duhamel formulation of \eqref{KDV3}
for $v$ (here, we use the fact that $v|_{t = 0} = 0$), we have 
\begin{align}
\begin{split}
& T^\frac 12\|\dt \eta_{_T}\cdot v\|_{C([-T, 2T]; H^{-\frac{3}{2}-\eps}_x)}
\les T^{\frac 12}
\|\eta'\|_{L^\infty}
 \|(v+\Phi)^2\|_{C([-T, 2T]; H^{-\frac 12 - \eps}_x)}\\
& \quad \les T^\frac 12 
\|\eta'\|_{L^\infty} \Big(\|v\|_{C([-1, 1]; L^2_x)}^2
+  \|\Phi\|_{C([-1, 1]; L^2_x)}^2\Big)\\
& \quad \les \dl, 
\end{split}
\label{ZH9}
\end{align}

\noi
where the last step follows
from 
taking 
$T = T\big(\o, \dl, \|u_0\|_{L^2}, \|\phi\|_{\HS(L^2; L^2)}\big) > 0$ 
sufficiently small. 
Hence, from \eqref{ZH7}, 
\eqref{ZH8}, and \eqref{ZH9}, 
we obtain 
\begin{align}
\| v \|_{X^{-\frac 32 - \eps, 1}(T)} \les \dl.
\label{ZH10}
\end{align}

\noi
Therefore, by interpolating \eqref{ZH5} and \eqref{ZH10}, 
we conclude that 
\begin{align}
\|v\|_{X^{- \frac 32\ta-\eps, \ta}(T)} \les \dl
\label{ZH11}
\end{align}

\noi
for any $0 \le \ta \le 1$ and $\eps > 0$, 
provided that 
$T = T\big(\o, \dl, \|u_0\|_{L^2}, \|\phi\|_{\HS(L^2; L^2)}\big) > 0$ 
sufficiently small.

Given small $\eps > 0$, 
let  $b= \frac 12 + \eps$ and $q = \frac 2{1-3\eps}$  be as in \eqref{lwp0}.
Then, define the $\cX^\eps$- and $\cY^\eps$-norms by 
\begin{align}
\begin{split}
\|u\|_{\cX^\eps}
& =  \|u \|_{X^{-\frac 34 - 2\eps, b}}
+ \|u \|_{X^{-\frac 38 - \frac 12 \eps, \frac 14}}
+  \| u\|_{X^{ - \frac 34 - \frac {13}2\eps, \frac 12 + 4\eps}}, \\
\|u \|_{\cY^\eps} 
& = 
 \| u\|_{X^{0, b}_q}
+  \| u\|_{X^{0, \frac 12 - \eps}}. 
\end{split}
\label{XX1}
\end{align}

\noi
Note that the $X^{s, b}$-norms appearing in the definition of the $\cX^\eps$-norm correspond
to the $X^{- \frac 32\ta-\eps, \ta}$-norm in~\eqref{ZH11} (where we replace $\eps$ by $\frac 12 \eps$)
with $\ta = b$, $\ta = \frac14$, and $\ta = \frac 12 + 4\eps$, respectively.
As in~\eqref{Xsb2}, given $T > 0$, 
we define the local-in-time version of the $\cX^\eps$-norm 
by setting
 \begin{align}
\| u\|_{\cX^\eps(T) }=\inf\big\{\| v\|_{\cX^\eps}: v|_{[0,T]}=u\big\}, 
\label{XX2}
\end{align}
     
\noi
where the infimum is taken over all extensions $v$ of $u$
from $[0, T]$ to $\R$.
We  define the $\cY^\eps(T)$-norm in an analogous manner.
In view of \eqref{Xsb3}, we have 
\begin{align}
\| u\|_{\cY^\eps(T) }
\sim \| \ind_{[0, T]} \cdot u \|_{\cY^\eps}.
\label{XX3}
\end{align}

\subsection{Bilinear estimate}

We first recall the following bilinear estimate
from Zhou's work; see
\cite[Theorems 1.1 and 2.1]{Zhou}.

\begin{lemma}\label{LEM:bilin2}
Given small $\eps > 0$, let $b = \frac 12 + \eps$ as in \eqref{lwp0}.
Then, we have 
\begin{align*}
& \| \dx (u_1u_2)\|_{X^{-\frac 34 - 2\eps, b-1}(T)}
  \les 
\|u_1 \|_{X^{-\frac 34 - 2\eps, b}(T)}
\|u_2 \|_{\cX^\eps(T)}
\end{align*}

\noi
for any $T > 0$,  
where 
$\cX^\eps(T)$ is as in \eqref{XX1} and \eqref{XX2}.

\end{lemma}

In the current stochastic setting, we also need the following bilinear estimate.

\begin{lemma}\label{LEM:bilin3}
Given $\eps > 0$, let
 $b$ and $q$ be as in \eqref{lwp0}.
Then, we have 
\begin{align}
\| \dx (u_1u_2)\|_{X^{-\frac 34 - 2\eps, b-1}(T)}
\les 
\|u_1 \|_{X^{-\frac 34 - 2\eps, b}(T)}
 \| u_2\|_{\cY^\eps(T)}
\label{bd2}
\end{align}

\noi
for any $T > 0$,  
where 
$\cY^\eps(T)$ is as in \eqref{XX1} and \eqref{XX3}.

\end{lemma}

\begin{proof}
As in the proof of Proposition \ref{PROP:bilin1}, 
it suffices to prove \eqref{bd2}
without a time restriction.
We split the proof into two cases:
$|\xi| \ges |\xi_1|$ and $|\xi| \ll |\xi_1|$.
When $|\xi| \ges |\xi_1|$, 
we prove
\begin{align}
\| \dx (u_1u_2)\|_{X^{-\frac 34 - 2\eps, b-1}}
\les 
\|u_1 \|_{X^{-\frac 34 - 2\eps, b}}
 \| u_2\|_{X^{0, b}_q}, 
\label{bd3}
\end{align}

\noi
while, when $|\xi| \ll |\xi_1|$, 
we prove
\begin{align}
\| \dx (u_1u_2)\|_{X^{-\frac 34 - 2\eps, b-1}}
\les 
\|u_1 \|_{X^{-\frac 34 -2 \eps, b}}
+  \| u_2\|_{X^{0, \frac 12 - \eps}}.
\label{bd4}
\end{align}

\medskip

\noi
$\bul$ {\bf Case 1:} $|\xi| \ges |\xi_1|$.\\
\indent
The bound \eqref{bd3} follows once we prove
\begin{align}
\bigg\|
\intt_{\substack{\tau = \tau_1 + \tau_2\\\xi = \xi_1 + \xi_2}}
\frac{\ind_{|\xi| \ges |\xi_1|} \cdot |\xi|} {\s^{1-b}}
\frac{ f_1(\tau_1, \xi_1)f_2(\tau_2, \xi_2) d\tau_1 d\xi_1}
{\s_1^{b}
\s_2^{b}}
\bigg\|_{L^2_{\tau, \xi}}
\les 
 \|f_1\|_{L^2_{\tau, \xi}}
 \|f_2\|_{L^2_\xi L^q_\tau}, 
\label{Y1}
\end{align}

\noi
where $\s$, $\s_1$, and $\s_2$ are as in \eqref{bilin4}.
For simplicity of notation, 
we  drop
the conditions 
$\tau = \tau_1 + \tau_2$, 
 $\xi = \xi_1 + \xi_2$, 
and $|\xi| \ges |\xi_1|$
in the following 
but it is understood that these conditions hold.

By applying H\"older's inequality in $\tau$
(with $\frac 12 = \frac {3\eps}{2} + \frac 1q$, 
where $q$ is as in \eqref{lwp0}) and proceeding as in 
\eqref{bilin4a}, we have 
\begin{align}
\begin{split}
\text{LHS of }\eqref{Y1}
& \les 
\bigg\|
\intt_{\substack{\tau = \tau_1 + \tau_2\\\xi = \xi_1 + \xi_2}}
\frac{\ind_{|\xi| \ges |\xi_1|} \cdot |\xi|} {\s^{1-b-2\eps}}
\frac{ f_1(\tau_1, \xi_1)f_2(\tau_2, \xi_2) d\tau_1 d\xi_1}
{\s_1^{b}
\s_2^{b}}
\bigg\|_{L^2_{\xi}L^q_\tau}\\
& \le 
\bigg\|
\wt M_{\tau, \xi}\, 
\Big\| 
 \prod_{j = 1}^2
 f_1(\tau_1, \xi_1)
  f_2(\tau - \tau_1, \xi - \xi_1)
\Big \|_{L^2_{\tau_1, \xi_1}} 
\bigg\|_{L^2_\xi L^q_\tau}\\
& \le \sup_{\tau, \xi \in \R} \wt M_{\tau, \xi}\, 
 \|f_1\|_{L^2_{\tau, \xi}}
 \|f_2\|_{L^2_\xi L^q_\tau}, 
\end{split}
\label{Y2}
\end{align}

\noi
where $\wt M_{\tau, \xi}$ is given by 
\begin{align*}
\wt M_{\tau, \xi} = \bigg\|
\frac{ |\xi|} {\s^{1-b-2\eps}\s_1^{b}
\s_2^{b}}
\bigg\|_{L^2_{\tau_1, \xi_1}}.
\end{align*}

\noi
By choosing $\eps > 0$
sufficiently small with \eqref{lwp0}, we have  $1-b -2\eps = \frac 12 - 3\eps \ge \frac 14$, 
and thus 
it follows from  \cite[Lemma 2.4]{KPV96} that 
\begin{align}
\sup_{\tau, \xi \in \R}
\wt M_{\tau, \xi} < \infty.
\label{Y3}
\end{align}

\noi
Hence, the bound \eqref{Y1}
follows from \eqref{Y2} and \eqref{Y3}.

\medskip

\noi
$\bul$ {\bf Case 2:} $|\xi| \ll |\xi_1|$.\\
\indent
The bound
\eqref{bd4} follows once we prove 
\begin{align}
\bigg\|
\intt_{\substack{\tau = \tau_1 + \tau_2\\\xi = \xi_1 + \xi_2}}
\frac{
\ind_{|\xi| \ll |\xi_1|}\cdot
|\xi|\jb{\xi_1}^{\frac 34  + 2\eps}} {\jb{\xi}^{\frac 34 + 2\eps}}
\frac{ f_1(\tau_1, \xi_1)f_2(\tau_2, \xi_2) d\tau_1 d\xi_1}
{\s^{1-b}\s_1^{b}
\s_2^{\frac 12 - \eps}}
\bigg\|_{L^2_{\tau, \xi}}
\les \prod_{j = 1}^2 \|f_j\|_{L^2_{\tau, \xi}}, 
\label{Z1}
\end{align}

\noi
where $\s$, $\s_1$, and $\s_2$ are as in \eqref{bilin4}.

In this case, we have $|\xi_1|\sim |\xi_2|\gg |\xi|$.
Then, 
from the proof of \cite[Theorem 1.1]{KPV96}, we have
\begin{align}
\bigg\|
\intt_{\substack{\tau = \tau_1 + \tau_2\\\xi = \xi_1 + \xi_2}}
\frac{|\xi|\jb{\xi_1}^{\frac 32-2\eps_1}}
 {\jb{\xi}^{\frac 34-\eps_1}}
\frac{ f_1(\tau_1, \xi_1)f_2(\tau_2, \xi_2) d\tau_1 d\xi_1}
{\s^{1-b} \s_1^{b}
\s_2^{b}}
\bigg\|_{L^2_{\tau, \xi}}
\les \prod_{j = 1}^2 \|f_j\|_{L^2_{\tau, \xi}}
\label{Z2}
\end{align}

\noi
for small $\eps, \eps_1> 0 $ such that 
\begin{align}
3\eps < \eps_1.
\label{Z2a}
\end{align}

\noi
On the other hand, by Young's and Cauchy-Schwarz's inequalities with~\eqref{lwp0}, we have 
\begin{align}
\bigg\|
\intt_{\substack{\tau = \tau_1 + \tau_2\\\xi = \xi_1 + \xi_2}}
\frac{1}
 {\jb{\xi_1}^{\frac 12+\eps}}
\frac{ f_1(\tau_1, \xi_1)f_2(\tau_2, \xi_2) d\tau_1 d\xi_1}
{\s_1^{b}}
\bigg\|_{L^2_{\tau, \xi}}
\les \prod_{j = 1}^2 \|f_j\|_{L^2_{\tau, \xi}}.
\label{Z3}
\end{align}

We first consider the case $|\xi|\les 1$.
In this case, we have 
\begin{align}
\frac{|\xi|\jb{\xi_1}^{\frac 34  + 2\eps}} {\jb{\xi}^{\frac 34 + 2\eps}}
\sim 
|\xi|\jb{\xi_1}^{\frac 34  + 2\eps}
\quad \text{and}\quad
\frac{|\xi|\jb{\xi_1}^{\frac 32-2\eps_1}}
 {\jb{\xi}^{\frac 34-\eps_1}}
\sim
|\xi|\jb{\xi_1}^{\frac 32-2\eps_1}.
\label{Z4}
\end{align}

\noi
Let  $\ta  = \ta(\eps) \in (0, 1)$ such that 
\begin{align}
\frac 12 - \eps \ge (1- \ta) b = (1- \ta)\Big(\frac 12 + \eps \Big).
\label{Z4a}
\end{align}

\noi
By choosing $\eps > 0$ small, we can make $\ta > 0$ arbitrarily small.
We note that 
\begin{align}
\begin{split}
|\xi|\jb{\xi_1}^{\frac 34  + 2\eps}
& \le 
\big(|\xi|^{\frac 1{1-\ta}}\jb{\xi_1}^{\frac 32  - 2\eps_1}\big)^{1-\ta}
\big(\jb{\xi_1}^{-\frac 12 - \eps}\big)^\ta\\
& \les 
\big(|\xi|\jb{\xi_1}^{\frac 32  - 2\eps_1}\big)^{1-\ta}
\big(\jb{\xi_1}^{-\frac 12 - \eps}\big)^\ta
\end{split}
\label{Z5}
\end{align}

\noi
for $\ta, \eps, \eps_1 > 0$ sufficiently small, where we used $|\xi|\les 1$
at the second inequality.
Hence, in view of  \eqref{Z4}, \eqref{Z4a}, and \eqref{Z5}, 
we obtain the bound~\eqref{Z1} 
by  
interpolating 
\eqref{Z2} and  
\eqref{Z3}, 
provided that $\eps, \eps_1 > 0$ are sufficiently small, 
satisfying \eqref{Z2a}.

Next, we  consider the case $1\ll |\xi|\ll|\xi_1|$.
In this case, we have 
\begin{align}
\frac{|\xi|\jb{\xi_1}^{\frac 34  + 2\eps}} {\jb{\xi}^{\frac 34 + 2\eps}}
\sim 
\jb{\xi}^{ \frac 14  - 2\eps}
\jb{\xi_1}^{\frac 34  + 2\eps}
\quad \text{and}\quad
\frac{|\xi|\jb{\xi_1}^{\frac 32-2\eps_1}}
 {\jb{\xi}^{\frac 34-\eps_1}}
\sim 
\jb{\xi}^{\frac 14 + \eps_1}
\jb{\xi_1}^{\frac 32-2\eps_1}.
\label{Z6}
\end{align}

\noi
Since we have 
\begin{align}
\jb{\xi}^{ \frac 14  - 2\eps}
\jb{\xi_1}^{\frac 34  + 2\eps}
\le \big(\jb{\xi}^{\frac 14 + \eps_1}
\jb{\xi_1}^{\frac 32-2\eps_1}\big)^{1-\ta}
\big(\jb{\xi_1}^{-\frac 12 - \eps}\big)^\ta
\label{Z7}
\end{align}

\noi
for sufficiently small  $\ta > 0$, 
the bound \eqref{Z1} follows
from interpolating \eqref{Z2} and \eqref{Z3}
with \eqref{Z4a}, \eqref{Z7}, and \eqref{Z6}, 
provided that $\eps, \eps_1 > 0$ are sufficiently small, 
satisfying \eqref{Z2a}.
\end{proof}

\subsection{Proof of Theorem \ref{THM:2}}
\label{SUBSEC:4.3}

Given $u_0 \in L^2(\R)$ and $\phi \in \HS(L^2(\R); L^2(\R))$, 
let $u_1$ and $u_2$ be two solutions to \eqref{mild1}
on $[0, T]$ with $u_1|_{t = 0} = u_2|_{t = 0} = u_0$.
More precisely,
for $j = 1, 2$, 
there exists a set $\Si_j \subset \Om$ with $\PP(\Si_j) = 1$
such that, 
for  $\o \in \Si_j$, 
there exists a global solution $u_j
\in C(\R; L^2(\R))$ to \eqref{KDV1}
with a realization of the noise (and hence of 
the stochastic convolution $\Psi = \Psi(\o)$ in \eqref{mild1})
for $\o \in \Si_j$.
By setting $\Si = \Si_1 \cap \Si_2$, 
we have $\PP(\Si) = 1$
and, 
for each $\o \in \Si$, 
$u_1 =  u_1(\o)$
and 
$u_2 =  u_2(\o)$ 
are global solutions to \eqref{KDV1}
with $u_1|_{t = 0} = u_2|_{t = 0} = u_0$
and 
 the realization of the noise 
for $\o \in \Si$.
In the following, we fix $\o \in \Si$
and suppress the dependence on $\o \in \Si$.

Consider the following first order expansion:
\begin{align*}
u_j = \Phi + v_j, \quad j = 1, 2, 
\end{align*}

\noi
where $\Phi$ is as in \eqref{ZH3}, 
such that 
 $v_j$ satisfies \eqref{KDV3}
(where $v$ is replaced by $v_j$).
Then, from~\eqref{ZH11} (or rather the computation leading to \eqref{ZH11}), 
we have 
\begin{align}
\sup_{j = 1, 2}\|v_j\|_{\cX^\eps(T)} \le C_0  \dl
\label{ZZ0}
\end{align}

\noi
for some $C_0 > 0$, 
provided that 
$T = T\big(\o, \dl, \|u_0\|_{L^2}, \|\phi\|_{\HS(L^2; L^2)}\big) > 0$ is sufficient small.

The difference $w = v_1 - v_2$ satisfies the following equation:
\begin{align}
\begin{cases}
\dt w + \dx^3 w + \dx \big(w(v_1+v_2)\big) + 2 \dx (w \Phi)= 0\\
w|_{t = 0} = 0.
\end{cases}
\label{ZZ1}
\end{align}

\noi
By applying Lemma \ref{LEM:lin1}
to the Duhamel formulation of \eqref{ZZ1}
and then applying Lemmas~\ref{LEM:bilin2} and \ref{LEM:bilin3}, 
we have 
\begin{align}
\|w \|_{X^{-\frac 34 - 2\eps, b}(T)}
&
  \le C_1
\|w \|_{X^{-\frac 34 - 2\eps, b}(T)}
\bigg(\sum_{j = 1}^2 \|v_j \|_{\cX^\eps(T)}
+  \| \Phi \|_{\cY^\eps(T)}\bigg)
\label{ZZ2}
\end{align}

\noi
for any $T > 0$.
Given  small $\dl > 0$ with $C_0 C_1 \dl < 1$, 
it follows from 
\eqref{ZZ2},  \eqref{ZZ0},  and~\eqref{Xsb4}
that 
\begin{align}
\|w \|_{X^{-\frac 34 - 2\eps, b}(T)}
&
  \le C_0 C_1 \dl
\|w \|_{X^{-\frac 34 - 2\eps, b}(T)}, 
\label{ZZ3}
\end{align}

\noi
provided that 
$T = T\big(\o, \dl, \|u_0\|_{L^2}, \|\phi\|_{\HS(L^2; L^2)}\big) > 0$ is sufficient small.
Hence, 
we conclude from \eqref{ZZ3} that 
$w \equiv 0$ on the time interval $[0, T]$.
By iteratively applying this argument, 
we conclude that 
$w \equiv 0$ on $\R$
and hence $u_1 \equiv  u_2$.

\appendix

\section{Boundedness of 
the multiplication by a sharp cutoff function on 
 Fourier-Lebesgue and Sobolev spaces}\label{SEC:A}

In this appendix, we study boundedness
properties
of the
multiplication by a sharp cutoff function 
in the Fourier-Lebesgue spaces and Sobolev spaces.

\begin{lemma}\label{LEM:X2}
Given $1 < q < \infty$, let $0 \le b < \frac{1}{q'}$.
Then, we have
\begin{align}
\| \ind_I \cdot f \|_{\F L^{b, q}(\R)} \les \| f\|_{\F L^{b, q}(\R)}
\label{FL0a}
\end{align}

\noi
for any interval $I \subset \R$, 
where the implicit constant is independent of $I$.

\end{lemma}

When $q = 2$, \eqref{FL0a} reduces to 
the following well-known bound:
\begin{align*}
\| \ind_I \cdot f \|_{H^b(\R)} \les \| f\|_{H^b(\R)}
\end{align*}

\noi
for $0 \le b < \frac 12$.
In this case, we
can appeal
to the physical space characterization
of $H^b(\R)$ and an interpolation;
see the proof of \cite[Lemma 2.1]{DD2}.
However, for general $q\ne 2$, 
the Fourier-Lebesgue norm
does not have such a physical space characterization.
Our proof is based on viewing 
$\ind_I$ 
as a suitable linear combination of the identity operator
and the Hilbert transform
on the Fourier side.
Then, we invoke the boundedness
of the Hilbert transform on $L^q(\R)$ with an $A_q$-weight.

\begin{proof}

We   define
the homogeneous Fourier-Lebesgue space $\dot {\F L}^{b, q}$
to be  the completion of the Schwartz class under the semi-norm:
\begin{align*}
\|f\|_{\dot{\F L}^{b, q}} = \big\| |\tau|^b \ft f(\tau)\big\|_{L^q}.
\end{align*}

\noi
We first prove an analogue of \eqref{FL0a} in Lemma \ref{LEM:X2}
for the homogeneous Fourier-Lebesgue spaces.
In fact, we prove the following bound;
given  $1 < q < \infty$ and $-\frac 1q <   b < \frac{1}{q'}$, 
we have 
\begin{align}
\| \ind_I \cdot  f \|_{\dot {\F  L}^{b, q}} \les \| f\|_{ \dot {\F L}^{b, q}}
\label{FL3}
\end{align}

\noi
for any interval $I \subset \R$.
Namely, we will prove
\begin{align}
\int_\R |\ft {\ind_I f}(\tau)|^q w_{b, q}(\tau) d\tau \les 
\int_\R |\ft {f}(\tau)|^q w_{b, q}(\tau) d\tau,  
\label{FL4}
\end{align}

\noi
where $w_{b, q}(\tau) = |\tau|^{bq}$.

Recall the Hilbert transform $H$ is defined by
\begin{align*}
\ft{Hg} (t) = -i \sgn(t) \ft g(t), 
\end{align*}

\noi
where the signum function $\sgn$ is given by 
\begin{align*}
\sgn (t) = 
\begin{cases}
1, & \text{if } t > 0, \\
0, & \text{if } t = 0,\\
-1, & \text{if } t < 0.
\end{cases}
\end{align*}

\noi
Given an interval $I = [a,b] \subset \R$, 
we have 
\begin{align*}
\ind_{[a, b]} = \ind_{[a, \infty)}(t) - \ind_{[b, \infty)}(t)
= \frac{\sgn(t - a) - \sgn(t - b)}{2}
\end{align*}

\noi
except for $t = a, b$.
Thus, by a direct computation, we have
\begin{align}
\begin{split}
& 
\F(\ind_{[a, b]}\, f)(\tau) 
 = 
\F^{-1}(\ind_{[a, b]}\, f)(-\tau) \\
& \quad = \frac i 2\Big( e^{-ia\tau} H(e^{-ia \cdot} \F^{-1}(f))(-\tau)
-  e^{-ib\tau} H(e^{-ib \cdot} \F^{-1}(f))(-\tau)\Big).
\end{split}
\label{FL5}
\end{align}

In view of \eqref{FL5}, 
we see that the bound \eqref{FL4} (and hence \eqref{FL3})
follows once we show that the Hilbert transform is bounded on 
the weighted Lebesgue space $L^q(\R, w_{b, q})$.
Recall that the Hilbert transform is a Calder\'on-Zygmund operator
with a standard kernel (see \cite[Definition 7.4.1]{Gra}).
Hence, in view of   \cite[Theorem 7.4.6]{Gra}, 
we only need to check 
if the weight $w_{b, q}(\tau) = |\tau|^{bq}$
 is an $A_q$ weight.
It follows from   \cite[Example 7.17]{Gra} that 
for $1 < q < \infty$, 
the weight $w_{b, q} = |\tau|^{bq}$ is an $A_q$ weight if and only if
$-1 < bq < q-1$, namely
\[ -\frac 1q < b < \frac{1}{q'}.\]

\noi
This proves \eqref{FL4} and hence \eqref{FL3}.

When $b = 0$, \eqref{FL3} yields
\begin{align}
\| \ind_I \cdot f \|_{\F L^{0, q}} \les \| f\|_{\F  L^{0, q}}.
\label{FL6}
\end{align}

\noi
Then, 
the bound \eqref{FL0a} 
follows from \eqref{FL3} (with $0 \le b <  \frac{1}{q'})$
and \eqref{FL6}.
By a scaling argument, it is easy to see that 
the implicit constants in the estimates \eqref{FL3}
and \eqref{FL6} are independent of $I$.
As a consequence, the  implicit constant in  \eqref{FL0a} 
is also independent of $I$.
\end{proof}

As a corollary to Lemma \ref{LEM:X2}, we present the following boundedness
result on Sobolev spaces.
Given $1 \le q \le \infty$ and $b \in \R$, 
let   $W^{b, q}(\R)$ be the nonhomogeneous Sobolev space 
(= Bessel potential space)
defined via the norm{\rm :} 
\[ \|f \|_{ W^{b, q}(\R)} =  \big\| \jb{\nb}^b f \big\|_{L^q(\R)},\]

\noi
where $\jb{\nb}^b$ is the Bessel potential of order $-b$ given by 
the Fourier multiplier $\jb{\tau}^b$.

\begin{corollary}
\label{COR:Bdd}
Let $ 1 < q < \infty$ and $0 \le b <\max\big( \frac 1q,  \frac 1{q'} \big)$.
Then, we have 
\begin{align}
\|\ind_I \cdot  f\|_{W^{b,q}(\R)} \les \|f\|_{ W^{b,q}(\R)},
\label{BL0}
\end{align}

\noi
for any interval $I \subset \R$, 
where the implicit constant is independent of $I$.

\end{corollary}

We refer interested readers
to the monograph \cite{MS} on further studies
on 
Sobolev multipliers.

\begin{proof}
[Proof of Corollary \ref{COR:Bdd}]

Given small  $\eps > 0$, 
it follows from 
Lemma \ref{LEM:X2}
with $q = 2$ that 
\begin{align}
\| \ind_I \cdot f \|_{H^{\frac 12 - \eps} (\R)} \les \| f\|_{H^{\frac 12 - \eps}(\R)}, 
\label{BL1}
\end{align}

\noi
where the implicit constant is independent of $I$.
Given any $1 \le p \le \infty$, we also have the following trivial bound:
\begin{align}
\| \ind_I \cdot f \|_{L^p (\R)} \le \| f\|_{L^p(\R)}.
\label{BL2}
\end{align}

\noi
Then, the bound \eqref{BL0} follows
from interpolating
\eqref{BL1} and \eqref{BL2} (with $1 < p < \infty$);
see
\cite[Theorems~4.4.1 and~6.4.5]{BL}.
By taking  $p \to \infty$ and $\eps \to 0$, 
we obtain the range
 $0 \le b < \frac1q$,  
while the range $0 \le b < \frac 1{q'}$
follows from  taking  $p \to 1$ and $\eps \to 0$.
\end{proof}

\begin{ackno}\rm
T.O.~would like to thank  Yuzhao Wang
for a helpful discussion on the material in Appendix \ref{SEC:A}.
K.T.~would like to thank the School of Mathematics at the University of Edinburgh for its hospitality, where this
manuscript was prepared during his visit.
D.G.~and T.O.~were supported by the European Research Council (grant no.~864138 ``SingStochDispDyn").
T.O.~also 
acknowledges support from  
the NSFC (grant no.~W2531005).

\end{ackno}


\begin{thebibliography}{99}


\bibitem{BIT}
A.~Babin, A.~Ilyin, E.~Edriss, 
{\it On the regularization mechanism for the periodic Korteweg-de Vries equation}, 
Comm. Pure Appl. Math. 64 (2011), no. 5, 591--648. 

\bibitem{BOW}
E.~Ba\c{s}ako\u{g}lu, T.~Oh, Y.~Wang,
{\it  Sharp unconditional well-posedness of the 2-d periodic cubic hyperbolic nonlinear Schr\"odinger equation}, 
arXiv:2509.01650 [math.AP].

\bibitem{BO}
\'A.~B\'enyi, T.~Oh, 
{\it  Modulation spaces, Wiener amalgam spaces, and Brownian motions}, Adv. Math. 228 (2011), no. 5, 2943--2981. 


\bibitem{BL}
J.~Bergh, J.~L\"ofstr\"om, 
{\it Interpolation spaces. An introduction},  Grundlehren der Mathematischen Wissenschaften, No. 223. Springer-Verlag, Berlin-New York, 1976. x+207 pp.

\bibitem{BO93}
J.~Bourgain, 
{\it Fourier transform restriction phenomena for certain lattice subsets and applications to
nonlinear evolution equations. II. The KdV equation,} 
 Geom. Funct. Anal. 3 (1993), no. 3, 209--262. 
	

\bibitem{BO96}
J.~Bourgain,
{\it Invariant measures for the 2D-defocusing nonlinear Schr\"odinger equation}, 
Comm. Math. Phys. 176 (1996), no. 2, 421--445.


%
%

\bibitem{CLOZ}
A.~Chapouto, G.~Li, T.~Oh, T.~Zhao,
{\it  Shallow-water convergence of the intermediate long wave equation in $L^2$}, 
arXiv:2511.15905 [math.AP]. 

\bibitem{Chen}
X.~Chen, J.~Holmer, \emph{The derivation of the $\mathbb {T}^{3}$ energy-critical NLS from quantum many-body dynamics}, Invent. Math. 217 (2019), no. 2, 433--547.



\bibitem{CLO}
K.~Cheung, G.~Li, T.~Oh, 
{\it Almost conservation laws for stochastic nonlinear Schr\"odinger equations}, 
J. Evol. Equ. 21 (2021), no. 2, 1865--1894. 

\bibitem{CGKO}
J.~Chung, Z.~Guo, S.~Kwon,T.~Oh,  
{\it Normal form approach to global well-posedness of the quadratic derivative nonlinear Schr\"odinger equation on the circle}, Ann. Inst. H. Poincar\'e Anal. Non Lin\'eaire 34 (2017), 1273--1297.



\bibitem{CKSTT03}
J.~Colliander, M.~Keel, G.~Staffilani, H.~Takaoka, T.~Tao,
{\it Sharp global well-posedness for KdV and modified KdV on $\R$ and $\T$},
J. Amer. Math. Soc. 16 (2003), no. 3, 705--749.



\bibitem{CO}
J.~Colliander, T.~Oh, 
{\it  Almost sure well-posedness of the cubic nonlinear Schr\"odinger equation below $L^2(\T)$}, Duke Math. J. 161 (2012), no. 3, 367--414. 

\bibitem{DPD}
G.~Da Prato, A.~Debussche,
{\it Strong solutions to the stochastic quantization equations}, 
Ann. Probab. 31 (2003), no. 4, 1900--1916.



\bibitem{DZ}
 G.~Da Prato, J.~Zabczyk, 
 {\it Stochastic equations in infinite dimensions,} Second edition. Encyclopedia of Mathematics and its Applications, 152. Cambridge University Press, Cambridge, 2014. xviii+493 pp.


\bibitem{DD1}
A.~de Bouard,  A.~Debussche,
\textit{The stochastic nonlinear Schr\"odinger equation in $H^1$},
Stochastic Anal. Appl. 21 (2003), no. 1, 97--126.

\bibitem{DD2} 
A.~de Bouard, A.~Debussche, {\it The Korteweg-de Vries equation with multiplicative homogeneous noise,}
  Stochastic differential equations: theory and applications, Interdiscip. Math. Sci., 2, World Sci. Publ., Hackensack, NJ (2007),  113--133.



\bibitem{DDT1}
A.~de Bouard, 
A.~Debussche, Y.~Tsutsumi, 
{\it White noise driven Korteweg-de Vries equation}, J. Funct. Anal. 169 (1999), no. 2, 532--558.



\bibitem{DDT2}
A.~de Bouard, 
A.~Debussche, Y.~Tsutsumi, 
{\it Periodic solutions of the Korteweg-de Vries equation driven by white noise},
 SIAM J. Math. Anal. 36 (2004/05), no. 3, 815--855.









\bibitem{FOW}
J.~Forlano, T.~Oh, Y.~Wang,
{\it Stochastic nonlinear Schr\"odinger equation with almost space-time white noise}, J. Aust. Math. Soc. 109 (2020), no. 1, 44--67. 




\bibitem{Furioli}
G.~Furioli, F.~Planchon, E.~Terraneo, 
{\it Unconditional well-posedness for semilinear Schr\"odinger and wave equations in $H^s$},
 Harmonic analysis at Mount Holyoke (South Hadley, MA, 2001), 147--156,
Contemp. Math., 320, Amer. Math. Soc., Providence, RI, 2003. 



\bibitem{GTV}
J.~Ginibre, Y.~Tsutsumi, G.~Velo, 
{\it On the Cauchy problem for the Zakharov system},
 J. Funct. Anal. 151 (1997), no. 2, 384--436.



\bibitem{Gra}
L.~Grafakos, 
{\it Classical Fourier analysis}, Third edition. Graduate Texts in Mathematics, 249. Springer, New York, 2014. xviii+638 pp. 


\bibitem{GOSW}
D.~Greco, T.~Oh, P.~Sosoe, Y.~Wang,
{\it  Normal form approach to unconditional well-posedness of the periodic stochastic Korteweg-de Vries equation with an additive noise}, preprint.


\bibitem{Grun1}
A.~Gr\"unrock, 
{\it An improved local well-posedness result for the modified KdV equation},
 Int. Math. Res. Not. 2004, no. 61, 3287--3308. 

\bibitem{Grun2}
A.~Gr\"unrock, 
{\it Bi- and trilinear Schr\"odinger estimates in one space dimension with applications to cubic NLS and DNLS}, Int. Math. Res. Not. 2005, no. 41, 2525--2558.

\bibitem{GH}
A.~Gr\"unrock, S.~Herr, 
{\it Low regularity local well-posedness of the derivative nonlinear Schr\"odinger equation
with periodic initial data}, SIAM J. Math. Anal. 39 (2008), no. 6, 1890--1920.


\bibitem{Guo}
Z.~Guo,
{\it  Global well-posedness of Korteweg-de Vries equation in $H^{-3/4}(\R)$}, 
J. Math. Pures Appl. 91 (2009), no. 6, 583--597.



\bibitem{GKO}
Z.~Guo, S.~Kwon, T.~Oh, \emph{Poincar{\'e}-Dulac normal form reduction for unconditional well-posedness of the periodic cubic NLS}, Commun. Math. Phys. 322 (2013), no. 1, 19--48.

%
%




\bibitem{HS}
S.~Herr, V.~Sohinger, \emph{Unconditional uniqueness results for the nonlinear Schr{\"o}dinger equation}, Commun. Contemp. Math. 21 (2019), no. 7, 28 pp.


\bibitem{KATO}
T.~Kato, {\it On nonlinear Schr\"odinger equations. II. $H^s$-solutions and unconditional
well-posedness,} J. Anal. Math. 67 (1995), 281--306.



\bibitem{KT1}
T.~Kato, K.~Tsugawa, 
{\it Cancellation properties and unconditional well-posedness for the fifth order KdV type equations with periodic boundary condition}, Partial Differ. Equ. Appl. 5 (2024), no. 3, Paper No. 18, 55 pp.


\bibitem{KT2}
T.~Kato, K.~Tsugawa, 
{\it 
Cancellation properties and unconditional well-posedness for fifth order modified KdV type equations with periodic boundary conditions}, J. Differential Equations 450 (2026), Paper No. 113736, 70 pp.


\bibitem{KPV93}
C.E.~Kenig, G.~Ponce, L.~Vega,
{\it The Cauchy problem for the Korteweg-de Vries equation in Sobolev spaces of negative indices},
Duke Math. J. 71 (1993), no. 1, 1--21. 

\bibitem{KPV96}
C.E.~Kenig, G.~Ponce, L.~Vega,
{\it A bilinear estimate with applications to the KdV equation},
J. Amer. Math. Soc. 9 (1996), no. 2, 573--603.


\bibitem{KOPV}
R.~Killip, O.~Pocovnicu, T.~Oh, M.~Vi\c{s}an, 
{\it Global well-posedness of the Gross-Pitaevskii and cubic-quintic nonlinear Schr\"odinger equations with non-vanishing boundary conditions}, Math. Res. Lett. 19 (2012), no. 5, 969--986. 


\bibitem{KV}
R.~Killip, M.~Vi\c{s}an, 
{\it KdV is well-posed in $H^{-1}$}, Ann. of Math.  190 (2019), no. 1, 249--305. 

\bibitem{Kishi}
N.~Kishimoto,
{\it Well-posedness of the Cauchy problem for the Korteweg-de Vries equation at the critical regularity},
Differential Integral Equations 22 (2009), no. 5-6, 447--464.


\bibitem{Kishi21}
N.~Kishimoto, \emph{Unconditional local well-posedness for periodic NLS}, 
 J. Differential Equations 274 (2021), 766--787. 


\bibitem{KO}
S.~Kwon, T.~Oh, 
{\it  On unconditional well-posedness of modified KdV}, Int. Math. Res. Not. (2012), no. 15, 3509--3534. 


\bibitem{MS}
V.~Maz'ya, T.~Shaposhnikova, 
{\it Theory of Sobolev multipliers. With applications to differential and integral operators},  Grundlehren der mathematischen Wissenschaften [Fundamental Principles of Mathematical Sciences], 337. Springer-Verlag, Berlin, 2009. xiv+609 pp.

\bibitem{McK}
H.~P.~McKean,
{\it Statistical mechanics of nonlinear wave equations. IV. Cubic Schr\"odinger}, 
Comm. Math. Phys. 168 (1995), no. 3, 479--491.
{\it Erratum: ``Statistical mechanics of nonlinear wave equations. IV. Cubic Schr\"odinger''}, 
Comm. Math. Phys. 173 (1995), no. 3, 675.


\bibitem{Moli}
L.~Molinet,
{\it A note on ill posedness for the KdV equation},
Differential Integral Equations 24 (2011), no. 7-8, 759--765.


\bibitem{NTT}
K.~Nakanishi, H.~Takaoka, Y.~Tsutsumi, 
{\it Counterexamples to bilinear estimates related with the KdV equation and the nonlinear Schr\"odinger equation}, IMS Conference on Differential Equations from Mechanics (Hong Kong, 1999). Methods Appl. Anal. 8 (2001), no. 4, 569--578. 


\bibitem{Oh4}
T.~Oh, {\it Periodic stochastic Korteweg-de Vries equation with additive space-time white noise}, Anal. PDE 2 (2009), no. 3, 281--304. 


\bibitem{OQS}
T.~Oh, J.~Quastel, P.~Sosoe,
{\it  Global dynamics for the stochastic KdV equation with white noise as initial data},
 Trans. Amer. Math. Soc. Ser. B 11 (2024), 420--460. 



\bibitem{OW2}
T.~Oh, Y.~Wang, 
\emph{Normal form approach to the one-dimensional periodic cubic nonlinear Schr{\"o}dinger equation in almost critical Fourier--Lebesgue spaces}, J. Anal. Math. 143 (2021), no. 2, 723--762.






\bibitem{Tsu1}
K.~Tsugawa, 
{\it Global well-posedness for the KdV equations on the real line with low regularity forcing terms},
 Commun. Contemp. Math. 8 (2006), no. 5, 681--713. 



\bibitem{Zhou}
Y.~Zhou, 
{\it Uniqueness of weak solution of the KdV equation},
 Internat. Math. Res. Notices 1997, no. 6, 271--283.





	
	

	

	
	
\end{thebibliography}
\end{document}